\documentclass[12pt]{amsart}

\usepackage{amsxtra,amssymb,amsthm,amsmath,amscd,url,epsfig}
\usepackage[utf8]{inputenc}
\usepackage{fullpage}
\renewcommand{\leq}{\leqslant}
\renewcommand{\geq}{\geqslant}

\def\stacksum#1#2{{\stackrel{{\scriptstyle #1}}
{{\scriptstyle #2}}}}

\newcommand{\Cc}{\mathbf{C}}

\newcommand{\Zz}{\mathbf{Z}}

\newcommand{\Rr}{\mathbf{R}}

\newcommand{\Qq}{\mathbf{Q}}

\newcommand{\proba}{\text{\boldmath$P$}}
\newcommand{\expect}{\text{\boldmath$E$}}

\newcommand{\charfun}{\mathbf{1}}

\newcommand{\mods}[1]{\,(\mathrm{mod}\,{#1})}

\newcommand{\ra}{\rightarrow}
\newcommand{\lra}{\longrightarrow}

\newcommand{\fleche}[1]{\stackrel{#1}{\lra}}

\newcommand{\eps}{\varepsilon}

\newcommand{\friable}{\Psi}

\DeclareMathSymbol{\gena}{\mathord}{letters}{"3C}
\DeclareMathSymbol{\genb}{\mathord}{letters}{"3E}

\def\sumb{\mathop{\sum \Bigl.^{\flat}}\limits}
\def\multsum{\mathop{\sum\cdots\sum}\limits}

\def\sums{\mathop{\sum \Bigl.^{*}}\limits}

\theoremstyle{plain}
\newtheorem{theorem}{Theorem}

\newtheorem{proposition}[theorem]{Proposition}

\theoremstyle{remark}
\newtheorem{remark}[theorem]{Remark}

\newtheorem*{rem}{Remark}
\newtheorem*{rems}{Remarks}

\theoremstyle{definition}

\newtheorem*{example*}{Example}

\begin{document}

\title{Amplification arguments for large sieve inequalities}
\author{E. Kowalski}

 \address{ETH Z\"urich -- D-MATH \\
  R\"{a}mistrasse 101\\ 8092 Z\"{u}rich, Switzerland}

\email{kowalski@math.ethz.ch}

\thanks{This work was completed while on sabbatical leave at the
  Institute for Advanced Study (Princeton, NJ). This material is based
  upon work supported by the National Science Foundation under
  agreement No. DMS-0635607. Any opinions, findinds and conclusions or
  recommendations expressed in this material are those of the
  author(s) and do not necessarily reflect the views of the National
  Science Foundation. }


\begin{abstract}
We give a new proof of the arithmetic large sieve inequality
  based on an amplification argument, and use a similar method to
  prove a new sieve inequality for classical holomorphic cusp forms. A
  sample application of the latter is also given.
\end{abstract}

\maketitle

\section{The classical large sieve}

The classical arithmetic large sieve inequality states that, for any
real numbers $N$, $Q\geq 1$, any choice of subsets $\Omega_p\subset
\Zz/p\Zz$ for primes $p\leq Q$, we have
\begin{equation}\label{eq-als}
|\{n\leq N\,\mid\, n\mods{p}\notin\Omega_p\text{ for } p\leq Q\}|
\leq \frac{\Delta}{H}
\end{equation}
where
$$
H=\sumb_{q\leq Q}{\prod_{p\mid q}{\frac{|\Omega_p|}{p-|\Omega_p|}}},
$$
and $\Delta$ is any constant for which the ``harmonic'' large sieve
inequality holds: for any complex numbers $(a_n)$, we have
\begin{equation}\label{eq-hls}
\sum_{q\leq Q}{\sums_{a\mods{q}}{\Bigl|\sum_{n\leq N}{
a_ne\Bigl(\frac{an}{q}\Bigr)
}\Bigr|^2}}
\leq \Delta \sum_{n\leq N}{|a_n|^2},
\end{equation}
the notation $\sumb$ and $\sums$ denoting, respectively, a sum over
squarefree integers, and one over integers coprime with the (implicit)
modulus, which is $q$ here.
\par
By work of Montgomery-Vaughan and Selberg, it is known that one can
take 
$$
\Delta=Q^2-1+N
$$ 
(see, e.g.,~\cite[Th. 7.7]{ik}).
\par
There are a number of derivations of~(\ref{eq-als})
from~(\ref{eq-hls}); for one of the earliest,
see~\cite[Ch. 3]{montgomery-2}. The most commonly used is probably the
argument of Gallagher involving a ``submultiplicative'' property of
some arithmetic function (see, e.g.,~\cite[\S 2.2]{sieve} for a very
general version).
\par
We will show in this note how to prove~(\ref{eq-als}) quite
straightforwardly from the \emph{dual} version of the harmonic large
sieve inequality: $\Delta$ is also any constant for which
\begin{equation}\label{eq-dls}
\sum_{n\leq N}{\Bigl|
\sum_{q\leq Q}{\sums_{a\mods{q}}{
\beta(q,a)e\Bigl(\frac{an}{q}\Bigr)}}\Bigr|^2}
\leq \Delta \sum_{q\leq Q}{\sums_{a\mods{q}}{|\beta(q,a)|^2}},
\end{equation}
holds for arbitrary complex numbers $(\beta(q,a))$. This is of some
interest because, quite often,\footnote{\ But not always --
  Gallagher's very short proof, found e.g. in~\cite[Th. 1,
  p. 549]{montgomery}, proceeds directly.} the
inequality~(\ref{eq-hls}) is proved by duality from~(\ref{eq-dls}),
and because, in recent generalized versions of the large sieve
(see~\cite{sieve}), it often seems that the analogue of~(\ref{eq-dls})
is the most natural inequality to prove -- or least, the most easily
accessible. So, in some sense, one could dispense entirely
with~(\ref{eq-hls}) for many applications! In particular, note that
both known proofs of the optimal version with $\Delta=N-1+Q^2$ proceed
by duality.
\par
Note that some ingredients of many previous proofs occur in this new
argument. Also, there are other proofs of~(\ref{eq-als}) working
directly from the inequality~(\ref{eq-dls}) which can be found in the
older literature on the large sieve, usually with explicit connections
with the Selberg sieve (see the references to papers of Huxley,
Kobayashi, Matthews and Motohashi in~\cite[p.  561]{montgomery}),
although none of those that the author has seen seems to give an
argument which is exactly identical or as well motivated. Also, traces
of this argument appear earlier in some situations involving modular
forms, e.g., in~\cite{duke-kowalski}. In Section~\ref{sec-2}, we will
use the same method to obtain a new type of sieve inequality for
modular forms; in that case, it doesn't seem possible to adapt easily
the classical proofs.
\par
Indeed, maybe the most interesting aspect of our proof is that it is
very easy to motivate. It flows very nicely from an attempt to improve
the earlier inequality
\begin{equation}\label{eq-rls}
|\{n\leq N\,\mid\, n\mods{p}\notin\Omega_p\text{ for } p\leq Q\}|
\leq \frac{\Delta}{K},\quad\quad
K=\sum_{p\leq Q}{\frac{|\Omega_p|}{p}},
\end{equation}
of R\'enyi, which is most easily proved using~(\ref{eq-dls}) instead
of~(\ref{eq-als}), as in~\cite[\S 2.4]{sieve}. 
\par
We will explain this quite leisurely; one could be much more concise
and direct (as in Section~\ref{sec-2}).
\par
Let 
$$
S=\{n\leq N\,\mid\, n\mods{p}\notin\Omega_p\text{ for } p\leq Q\},
$$
be the sifted set; we wish to estimate from above the cardinality of
this finite set.  From~(\ref{eq-dls}), the idea is to find an
``amplifier'' of those integers remaining in the sifted set, i.e., an
expression of the form
$$
A(n)=\sum_{q\leq Q}{\sums_{a\mods{q}}{
\beta(q,a)e\Bigl(\frac{an}{q}\Bigr)}}
$$
which is \emph{large} (in some sense) when $n\in S$. Then an estimate
for $|S|$ follows from the usual Chebychev-type manoeuvre.
\par
To construct the amplifier $A(n)$, we look first at a single prime
$p\leq Q$. If $n\in S$, we have $n\mods{p}\notin\Omega_p$. If we
expand the characteristic function of $\Omega_p$ in terms of additive
characters,\footnote{\ We use this specific basis to
  use~(\ref{eq-dls}), but any orthonormal basis containing the
  constant function $1$ would do the job, as in~\cite{sieve}.} we
have then
$$
0=\charfun_{\Omega_p}(n)=\sum_{a\mods{p}}{\alpha(p,a)
e\Bigl(\frac{an}{p}\Bigr)},\quad\quad\quad
\alpha(p,a)=\frac{1}{p}\sum_{x\in\Zz/p\Zz}
{\charfun_{\Omega_p}(x)e\Bigl(\frac{ax}{p}\Bigr)},
$$
and the point is that the contribution of the constant function
($0$-th harmonic) is, indeed, relatively ``large'', because it is
$$
\alpha(p,0)=\frac{|\Omega_p|}{p},
$$
and exactly reflects the probability of a random element being in
$\Omega_p$. Thus for $n\mods{p}\notin \Omega_p$, we have
\begin{equation}\label{eq-one-p}
\sums_{a\mods{p}}{\beta(p,a)e\Bigl(\frac{an}{p}\Bigr)}=c_p
\end{equation}
with 
$$
c_p=\frac{|\Omega_p|}{p},\quad\quad
\beta(p,a)=-\alpha(p,a).
$$
\par
If we only use the contribution of the primes in~(\ref{eq-dls}), and
the amplifier
$$
A(n)=\sum_{p\leq  Q}
{\sum_{a\mods{p}}{\beta(p,a)e\Bigl(\frac{an}{p}\Bigr)}},
$$
then by~(\ref{eq-dls}), we get
$$
\sum_{n\in S}{|A(n)|^2}\leq \sum_{n\leq N}{|A(n)|^2}\leq \Delta
\sum_{p\leq Q}{\sum_{a\mods{p}} {|\beta(p,a)|^2}}.
$$
\par
For $n\in S$, the size of the amplifier is
$$
|A(n)|^2=\Bigl|\sum_{p\leq  Q}
{\sum_{a\mods{p}}{\beta(p,a)e\Bigl(\frac{an}{p}\Bigr)}}\Bigr|^2
=\Bigl|\sum_{p\leq Q}{c_p}\Bigr|^2=K^2,
$$
by~(\ref{eq-one-p}), while on the other hand, by applying the Parseval
identity in $\Zz/p\Zz$, we get
\begin{align*}
\sum_{p\leq Q}{\sum_{a\mods{p}}
{|\beta(p,a)|^2}}&=
\sum_{p\leq Q}{\Bigl(\frac{1}{p}\sum_{x\in\Zz/p\Zz}{
|\charfun_{\Omega_p}(x)|^2}-\alpha(p,0)^2\Bigr)}\\
&=\sum_{p\leq Q}{c_p(1-c_p)}\leq K.
\end{align*}
\par
So we obtain 
$$
K^2|S|\leq \Delta K,
$$
i.e., exactly R\'enyi's inequality~(\ref{eq-rls}), by this technique.
\par
To go further, we must exploit all the squarefree integers $q\leq Q$
(and not only the primes) to construct the amplifier. This is most
easily described using the Chinese Remainder Theorem to write
$$
\Zz/q\Zz\simeq \prod_{p\mid q}{\Zz/p\Zz},\quad\quad
(\Zz/q\Zz)^{\times}\simeq \prod_{p\mid q}{(\Zz/p\Zz)^{\times}},
$$
and putting together the amplifiers modulo primes $p\mid q$: if $n\in
S$ then $n\mods{p}\notin \Omega_p$ for all $p\mid q$, and hence
multiplying out~(\ref{eq-one-p}) over $p\mid q$, we find constants
$\beta(q,a)\in \Cc$, defined for $(a,q)=1$ (because $\beta(p,a)$ is
defined for $a$ coprime with $p$), such that
$$
\sums_{a\mods{q}}{\beta(q,a)e\Bigl(\frac{an}{q}\Bigr)}=
\prod_{p\mid q}{c_p}.
$$
\par
Moreover, because the product decomposition of the Chinese Remainder
Theorem is compatible with the Hilbert space structure involved, we
have
$$
\sums_{a\mods q}{|\beta(q,a)|^2}=\prod_{p\mod q}{
\sums_{a\mods p}{|\beta(p,a)|^2}
}=\prod_{p\mid q}{c_p(1-c_p)}.
$$
\par
Arguing as before, we obtain from~(\ref{eq-dls}) -- using all
squarefree moduli $q\leq Q$ this time -- that
\begin{equation}\label{eq-weaker}
|S|\leq \Delta \frac{A}{B^2},
\end{equation}
with
$$
A=\sumb_{q\leq Q}{\prod_{p\mid q}{c_p(1-c_p)}},\quad\quad
B=\sumb_{q\leq Q}{\prod_{p\mid q}{c_p}}.
$$
\par
This is not quite~(\ref{eq-als}), but we have some flexibility to
choose another amplifier, namely, notice that this expression is not
homogeneous if we multiply the coefficients $\beta(q,a)$ by scalars
independent of $a$, and we can use this to find a better
inequality. Precisely, let
$$
\gamma(q,a)=\Bigl(\prod_{p\mid q}{\xi_p}\Bigr)\beta(q,a),
$$
where $\xi_p$ are arbitrary real coefficients.
\par
Then we have the new amplification property
$$
\sums_{a\mods{q}}{\gamma(q,a)e\Bigl(\frac{an}{q}\Bigr)}=
\prod_{p\mid q}{\xi_pc_p}
$$
with altered ``cost'' given by
$$
\sums_{a\mods q}{|\gamma(q,a)|^2}=
\prod_{p\mid q}{\xi_p^2c_p(1-c_p)},
$$
so that, arguing as before, we get
$$
|S|\leq \Delta \frac{A_1}{B_1^2}
$$
with
$$
A_1=\sumb_{q\leq Q}{\prod_{p\mid q}{\xi_p^2c_p(1-c_p)}},\quad\quad
B_1=\sumb_{q\leq Q}{\prod_{p\mid q}{\xi_pc_p}}.
$$
\par
By homogeneity, the problem is now to minimize a quadratic form
(namely $A_1$) under a linear constraint given by $B_1$. This is
classical, and is done by Cauchy's inequality: writing
$$
c_q=\prod_{p\mid q}{c_p},\quad\quad \tilde{c}_q=\prod_{p\mid
  q}{(1-c_p)},\quad\quad \xi_q=\prod_{p\mid q}{\xi_p}
$$
for ease of notation, we have
$$
B_1^2=\Bigl(\sumb_{q\leq Q}{\xi_qc_q}\Bigr)^2
\leq \Bigl(\sumb_{q\leq Q}{\xi_q^2c_q\tilde{c}_q}\Bigr)
\Bigl(
\sumb_{q\leq Q}{\frac{c_q}{\tilde{c}_q}}
\Bigr)=A_1H,
$$
with equality if and only if $\xi_p$ is proportional to
$$
\xi_p=\frac{1}{1-c_p}=\frac{p}{p-|\Omega_p|},
$$
in which case 
$$
c_p\xi_p=\xi_p^2c_p(1-c_p)=\frac{|\Omega_p|}{p-|\Omega_p|},
$$
and we get $A_1=B_1=H$, hence $|S|\leq \Delta H^{-1}$, which
is~(\ref{eq-als}). 
% But this choice is not just a random one: the
% problem is to optimize a quadratic form (namely $A_1$) with a linear
% constraint given by $B_1$, and the choice above is the best one. This
% is checked, as usual, with Cauchy's inequality: writing
% $$
% c_q=\prod_{p\mid q}{c_p},\quad\quad \tilde{c}_q=\prod_{p\mid
%   q}{(1-c_p)},\quad\quad \lambda_q=\prod_{p\mid q}{\xi_p}
% $$
% for ease of notation, we have
% $$
% B_1^2=\Bigl(\sumb_{q\leq Q}{\lambda_qc_q}\Bigr)^2
% \leq \Bigl(\sumb_{q\leq Q}{\lambda_q^2c_q\tilde{c}_q}\Bigr)
% \Bigl(
% \sumb_{q\leq Q}{\frac{c_q}{\tilde{c}_q}}
% \Bigr)=A_1H,
% $$
% with equality if $\xi_p$ is as given above, or in other words
% $$
% \frac{A_1}{B_1^2}\geq \frac{1}{H},
% $$
% with equality for this specific choice of $\xi_p$.
\par
\medskip
\par
\begin{rems}
  (1) The last optimization step is reminiscent of the Selberg sieve
  (see, e.g.,~\cite[p. 161, 162]{ik}). Indeed, it is well known that
  the Selberg sieve is related to the large sieve, and particularly
  with the dual inequality~(\ref{eq-dls}), as explained
  in~\cite[p. 125]{halberstam-richert}. Note however that the
  coefficients we optimize for, being of an ``amplificatory'' nature,
  and different from the coefficents $\lambda_d$ typically sought for
  in Selberg's sieve, which are akin to the Möbius function and of a
  ``mollificatory'' nature.
\par
(2) The argument does not use any particular feature of the classical
sieve, and thus extends immediately to provide a proof of the general
large sieve inequality of~\cite[Prop. 2.3]{sieve} which is directly
based on the dual inequality~\cite[Lemma 2.8]{sieve}; readers
interested in the formalism of~\cite{sieve} are encouraged to check
this.
\end{rems}

\begin{example*}
  What are the amplifiers above in some simple situations? In the case
  -- maybe the most important -- where we try to count primes, we then
  take $\Omega_p=\{0\}$ to detect integers free of small primes by
  sieving, and~(\ref{eq-one-p}) becomes
$$
\sums_{a\mods{p}}{\Bigl(-\frac{1}{p}\Bigr)\cdot
  e\Bigl(\frac{an}{p}\Bigr)}=\frac{1}{p},
$$
if $p\nmid n$. Then, for $q$ squarefree, the associated detector is
the identity
$$
\sums_{a\mods{q}}{\frac{\mu(q)}{q}
  e\Bigl(\frac{an}{q}\Bigr)}=\frac{1}{q},
$$
if $(n,q)=1$, or in other words, it amounts to the well-known formula
$$
\sums_{a\mods {q}}{e\Bigl(\frac{an}{q}\Bigr)}=\mu(q)
$$
for the values of a Ramanujan sum with coprime arguments. Note that in
this case, the optimization process above replaced
$c_p=\frac{1}{p}$ with
$$
\xi_pc_p=\frac{1}{p-1},
$$
which is not a very big change -- and indeed, for small sieves, the
bound~(\ref{eq-weaker}) is not far from~(\ref{eq-als}), and remains of
the right order of magnitude. 
\par
On the other hand, for an example in a large sieve situation, we can
take $\Omega_p$ to be the set of squares in $\Zz/p\Zz$. The
characteristic function (for odd $p$) is
$$
\charfun_{\Omega_p}(x)=\sum_{a\mods{p}}{
\tau(p,a)e\Bigl(\frac{ax}{p}\Bigr)}
$$
with coefficients given -- essentially -- by Gauss sums
$$
\tau(p,a)=\frac{1}{p}\Bigl(1+\frac{1}{2}\sums_{x\mods{p}}{
e\Bigl(\frac{ax^2}{p}\Bigr)
}\Bigr).
$$
\par
Then $c_p$ tends to $1/2$ as $p\ra +\infty$, while $\xi_pc_p$
tends to $1$. This difference leads to a discrepancy in the order of
magnitude of the final estimate: using standard results on bounds for
sums of multiplicative functions,~(\ref{eq-weaker}) and taking
$Q=\sqrt{N}$, we get
$$
|S|\ll \sqrt{N}(\log N)^{1/4},
$$
instead of $|S|\ll\sqrt{N}$ that follows from~(\ref{eq-als}).
\end{example*}

\section{Sieving for modular forms}\label{sec-2}

To illustrate the possible usefulness of the proof given in the first
section, we use the same technique to prove a new type of large sieve
inequality for classical (holomorphic) modular forms. The originality
consists in using known inequalities for Fourier coefficients (due to
Deshouillers-Iwaniec) as a tool to obtain a sieve where the cusp forms
are the objects of interest, i.e., to bound from above the number of
cusp forms of a certain type satisfying certain local conditions.
\par
%%%%%%%%%%%%%%%%%%%%%%%%%
\par
Let $k\geq 2$ be a fixed even integer. For any integer $q\geq 1$, let
$S_k(q)^*$ be the finite set of primitive holomorphic modular forms of
level $q$ and weight $k$, with trivial nebentypus (more general
settings can be studied, but we restrict to this one for
simplicity). We denote by
$$
f(z)=\sum_{n\geq 1}{n^{(k-1)/2}\lambda_f(n)e(nz)}
$$
the Fourier expansion of a form $f\in S_k(q)^*$ at the cusp at
infinity.
\par
We consider on this finite set the ``measure'' $\mu=\mu_q$ defined by
$$
\mu_q(\{f\})=\frac{(k-1)!}{(4\pi)^{k-1}\langle f,f\rangle},
$$
where $\langle \cdot,\cdot\rangle$ is the Petersson inner
product. This is the familiar ``harmonic weight'', and we denote 
\begin{equation}\label{eq-weight}
\expect_{q}(\alpha)=\sum_{f\in S_k(q)^*}{\mu_q(\{f\})\alpha(f)},\quad
\proba_{q}(\mathcal{P}\text{ is true})=\sum_{\stacksum{f\in
    S_k(q)^*}{\mathcal{P}(f)\text{ is true}}}{\mu_q(\{f\})}
\end{equation}
the corresponding averaging operator and ``probability'', for an
arbitrary property $\mathcal{P}(f)$ referring to the modular forms
$f\in S_k(q)^*$. (Note that it is only asymptotically that this is a
probability measure, as $q\ra +\infty$).
\par
Imitating the notation in~\cite[Ch. 1]{sieve}, we now denote by 
$$
\rho_p\,:\, \begin{cases}
S_k(q)^*\ra \Rr\\
f\mapsto \lambda_f(p),
\end{cases}
$$
the $p$-th Fourier coefficient maps, which we see as giving
``global-to-local'' data, similar to reduction maps modulo primes for
integers. If $d\geq 1$ is a squarefree integer coprime with $q$, we
denote
$$
\rho_d\,:\, \begin{cases}
S_k(q)^*\ra \Rr^{\omega(d)}\\
f\mapsto (\rho_p(f))_{p\mid d}=(\lambda_f(p))_{p\mid d},
\end{cases}
$$
which we emphasize is a \emph{tuple} of Fourier coefficients, that
should not be mistaken with the single number $\lambda_f(d)$.
\par
The basic relation with sieve is the following idea: provided $Q$ is
small enough, the $(\rho_p(f))_{p\leq Q}$ become equidistributed as
$q\ra +\infty$ for the product Sato-Tate measure
$$
\nu_Q=\bigotimes_{p\leq Q}{\mu_{ST}},
$$
where
$$
\mu_{ST}=\frac{1}{\pi}\charfun_{[-2,2]}(t)\sqrt{1-\frac{t^2}{4}}\ dt,
$$
and this is similar to the equidistribution of arithmetic sequences
like the integers or the primes modulo squarefree $d$, and the
independence due to the Chinese Remainder Theorem.
\par
The quantitative meaning of this principle is easy to describe if $Q$
is bounded (independently of $q$), but requires some care when it
grows with $q$. For our purpose, we express it as given by uniform
bounds for Weyl-type sums associated with a suitable orthonormal basis
of $L^2(\nu_{Q})$. The latter is easy to construct.  Indeed, recall
first the standard fact that the Chebychev polynomials $X_m$, $m\geq
0$, defined by
\begin{equation}\label{eq-cheby-pol}
X_m(2\cos\theta)=\frac{\sin((m+1)\theta)}{\sin\theta},
\quad\quad \theta\in [0,\pi],
\end{equation}
form an orthonormal basis of $L^2(\mu_{ST})$. Then standard arguments
show that for $Q\geq 2$ and $\nu_Q$ the measure above on
$[-2,2]^{\pi(Q)}$, the functions 
$$
\Lambda_d(x)=\prod_{p}{X_{m_p}(x_p)},\quad\text{ for all }
x=(x_p)_p\in [-2,2]^{\pi(Q)},
$$
defined for any $Q$-friable integer\footnote{\ I.e., integer only
  divisible by primes $\leq Q$.} $d\geq 1$, factored as
$$
d=\prod_{p\leq Q}{p^{m_p}},
$$
form an orthonormal basis of $L^2(\nu_Q)$. (In particular we have
$\Lambda_1=1$, the constant function $1$.)
\par
We have also the following fact which gives the link between this
orthonormal basis and our local data $(\rho_p)_p$: for any integer
$m\geq 1$ coprime with $q$ and divisible only by primes $p\leq Q$, and
any $f\in S_k(q)^*$, we have
\begin{equation}\label{eq-fact}
\Lambda_m(\rho_d(f))=\lambda_f(m),\quad\text{ where }\quad
d=\prod_{p\mid m}{p}.
\end{equation}
\par
This is simply a reformulation of the Hecke multiplicativity relations
between Fourier coefficients of primitive forms.

\begin{rem}
  Our situation is similar to that of classical sieve problems, where
  (in the framework of~\cite{sieve}) we have a set $X$ (with a finite
  measure $\mu$) and surjective maps $X\fleche{\rho_{\ell}} Y_{\ell}$
  with \emph{finite} target sets $Y_{\ell}$, each equipped with a
  probability density $\nu_{\ell}$, so that the equidistribution can
  be measured by the size of the remainders $r_{\ell}(y)$ defined by
$$
\mu(\rho_{\ell}^{-1}(y))=\mu(X)\nu_{\ell}(y)+r_{\ell}(y)
$$
and the independence by using finite sets
$m=\{\ell_1,\ldots,\ell_k\}$, and
$$
Y_m=\prod_{\ell\in m}{Y_{\ell}},\quad\quad
\rho_m=\prod_{\ell\in m}{\rho_{\ell}}\,:\, X\ra Y_m,\quad\quad
\nu_{m}(y_1,\ldots, y_k)=\nu_{\ell_1}(y_1)\cdots\nu_{\ell_k}(y_k),
$$
and looking at
$$
\mu(\rho_{m}^{-1}(y))=\mu(X)\nu_{m}(y)+r_{m}(y).
$$
\par
Here the compact set $[-2,2]$ requires the use of infinitely many
functions to describe an orthonormal basis. Another (less striking)
difference is that our local information lies in the same set $[-2,2]$
for \emph{all} primes, whereas classical sieves typically involve
reduction modulo primes, which lie in different sets.
\end{rem}

We now state the analogue, in this language, of the dual large sieve
inequality~(\ref{eq-dls}).

% . The first is a
% uniform version:

% \begin{proposition}\label{pr-equi-1}
%   With notation as above, for all $z\geq 1$, for all integers $d\mid
%   P_q(z)$ and all integers $m$ divisible only by primes dividing $d$,
%   we have
% $$
% \expect_q(\Lambda_m(\rho_d(f)))\ll mq^{-1},
% $$
% for $k\geq 2$, $q\geq 2$, where the implied constant depends only on
% $k$. 
% \end{proposition}

% \begin{proof}
% This is a direct application of the Petersson formula for Fourier
% coefficients of modular forms and the direct estimation of the
% remainder term (a series of Kloosterman sums and Bessel functions)
% using the Weil bound for Kloosterman sums.
% %%TODO: finish
% \end{proof}

% The second is a large sieve type inequality:

\begin{proposition}\label{pr-equi-2}
  With notation as above, for all $Q\geq 1$, all integers $N\geq 1$,
  all complex numbers $\alpha(m)$ defined for $m$ in the set
  $\friable_q(N,Q)$ of $Q$-friable integers $\leq N$ coprime with $q$,
  we have
\begin{equation}\label{eq-ls}
\expect_q\Bigl(\Bigl| \sum_{m\in \friable_q(N,Q)}
{\alpha(m)\Lambda_m(\rho_{d}(f))} \Bigr|^2\Bigr)\ll
(1+Nq^{-1})\sum_{m}{|\alpha(m)|^2},
\end{equation}
where the implied constant depends only on $k$ and $d$ on the
left-hand side is the radical $\prod_{p\mid m}{p}$.
\end{proposition}

\begin{proof}[Proof of Proposition~\ref{pr-equi-2}]
  This is in fact simply a consequence of one of the well-known large
  sieve inequalities for Fourier coefficients of cusp forms (as
  developped by Iwaniec and by Deshouillers--Iwaniec,
  see~\cite{desh-iwaniec}). The point is that because
  of~(\ref{eq-fact}), the left-hand side of~(\ref{eq-ls}) can be
  rewritten
\begin{align*}
S&=\sum_{f\in S_k(q)^*}{\mu_q(\{f\})
\Bigl| \sum_{m\in \friable_q(N,Q)}{
  \alpha(m)\lambda_f(m)} \Bigr|^2}
\\
&=\frac{(k-1)!}{(4\pi)^{k-1}}
\sum_{f\in S_k(q)^*}{
\Bigl| \sum_{m\in \friable_q(N,Q)}{
  \alpha(m)\frac{\lambda_f(m)}{\|f\|}} \Bigr|^2}
\end{align*}
\par
We can now enlarge this by positivity; remarking that 
$$
\{\frac{f}{\|f\|}\,\mid\, f\in S_k(q)^*\}
$$
can be seen as a subset of an orthonormal basis of the space $S_k(q)$
of cusp forms of weight $k$ and level $q$, and selecting any such
basis $\mathcal{B}_{k,q}\supset S_k(q)^*$, we have therefore
$$
S\leq \frac{(k-1)!}{(4\pi)^{k-1}}
\sum_{\varphi\in \mathcal{B}_{k,q}}{
\Bigl|
\sum_{m\leq N}{
\alpha(m)\lambda_{\varphi}(m)
}
\Bigr|^2
},
$$
where we put $\alpha(m)=0$ if $m\notin \friable_q(N,Q)$, and where the
$\lambda_{\varphi}(m)$ are the Fourier coefficients, so that
$$
\varphi(z)=\sum_{m\geq 1}{m^{(k-1)/2}\lambda_{\varphi}(m)e(nz)},
$$
(as earlier for Hecke forms). Now by the large sieve inequality
in~\cite[Theorem 7.26]{ik}, taking into account the slightly different
normalization,\footnote{\ The case $k=2$ requires adding a factor
  $\log N$.} we have
\begin{equation}\label{eq-d-i}
\frac{(k-1)!}{(4\pi)^{k-1}}
\sum_{\varphi\in \mathcal{B}_{k,q}}{
\Bigl|
\sum_{1\leq m\leq N}{
\alpha(m)\lambda_{\varphi}(m)
}
\Bigr|^2
}\ll (1+Nq^{-1})\sum_{m}{|\alpha(m)|^2}
\end{equation}
with an absolute implied constant, and this leads to~(\ref{eq-ls}). 
%%TODO: check k=2!!!
\end{proof}

\begin{remark}
  In terms of equidistribution (which are hidden in this proof), the
  basic statement for an individual prime $p$ is that
$$
\lim_{q\ra +\infty}{\expect_q(X_m(\rho_p(f)))}=0,
$$
for all $m\geq 1$. Such results are quite well-known and follow in
this case from the Petersson formula. There is an implicit version
already present in Bruggeman's work (see \cite[\S 4]{bruggeman}, where
it is shown that, on average, ``most'' Maass forms with Laplace
eigenvalue $\leq T$, satisfy the Ramanujan-Petersson conjecture), and
the first explicit result goes back to Sarnak~\cite{sarnak}, still in
the case of Maass forms.\footnote{\ This is the only result we know
  that discusses the issue of independence of the coefficients at
  various primes.}  Serre~\cite{serre} and Conrey, Duke and
Farmer~\cite{cdf} gave similar statements for holomorphic forms, and
Royer~\cite{royer} described quantitative versions in that case.%   The
% version below does not seem to have been written down explicitly
% before and its emphasis is slightly different, with relations to the
% general large-sieve framework developed in~\cite{sieve}. (Another
% difference, which does play a certain role, is that we use the
% Petersson formula, whereas the other references use the trace formula,
% which results in different limiting measures at the primes $p$ -- with
% the exception of Bruggeman, who uses his own analogue of the Petersson
% formula for Maass forms).
\end{remark}

We can now derive the analogues of the arithmetic
inequality~(\ref{eq-als}) and of R\'enyi's
inequality~(\ref{eq-rls}). The basic ``sieve'' questions we look at is
to bound from above the cardinality (or rather, $\mu_q$-measure) of
sets of the type
$$
S=\{f\in S_k(q)^*\,\mid\, \lambda_f(p)=\rho_p(f)\notin \Omega_p\text{
  for } p\leq Q,\ p\nmid q\},
$$
for $\Omega_p\subset [-2,2]$. Because the expansion of the
characteristic function of $\Omega_p$ in terms of Chebychev
polynomials involves infinitely many terms, we restrict to a simple
type of condition sets $\Omega_p$ of the following type:
\begin{equation}\label{eq-spec-omega}
\Omega_p=\{x\in [-2,2]\,\mid\, Y_p(x)\leq \beta_{p,0}-\delta_p\},
\end{equation}
where
$$
Y_p=\beta_{p,0}+\beta_{p,1} X_1+\cdots +\beta_{p,s} X_s
$$
is a real-valued polynomial and $\delta_p>0$ (the degree $s$ is
assumed to be the same for all $p$). Note that $\beta_{p,0}$ is the
$\mu_{ST}$-average of $Y_p$, so our sets $S$ are those where the
Fourier coefficients for $p\leq Q$ are ``away'' from the putative
average value according to the Sato-Tate measure.
\par
Denote also by
$$
\sigma_{p}^2= \sum_{1\leq i\leq
  s}{\beta_{p,i}^2}=\int_{-2}^2{Y_p^2d\mu_{ST}}
-\Bigl(\int_{-2}^2{Y_pd\mu_{ST}}\Bigr)^2,
$$
the variance of $Y_p$.
\par
Then the analogue of~(\ref{eq-rls}) is
\begin{equation}\label{eq-cor-1}
  \expect_q\Bigl(\Bigl(\sum_{p\leq
    Q}{(Y_p(\lambda_f(p))-\beta_{p,0})}\Bigr)^2\Bigr)
  \ll (1+Q^sq^{-1})\sum_{p\leq Q}{\sigma_p^2},
\end{equation}
where the implied constant depends only on $k$, and that
of~(\ref{eq-als}) is
\begin{equation}\label{eq-cor-2}
\proba_q\Bigl(Y_p(\lambda_f(p))\leq \beta_{p,0}-\delta_p\text{ for all }
p\leq Q\Bigr)
\ll (1+N^sq^{-1})H^{-1}
\end{equation}
where $\delta_p>0$, $N\geq 1$ is arbitrary and
$$
H=\sum_{m\in \friable_q(N,Q)}{
\prod_{p\mid m}{
\frac{\delta_p^2}
{\sigma_{p}^2}}},
$$
the implied constant depending again only on $k$.

To prove~(\ref{eq-cor-1}), we apply~(\ref{eq-ls}) with $N=Q^s$ and
$\alpha(m)=0$ unless $m=p^j$ with $1\leq j\leq s$ and $p\leq Q$,
$p\nmid q$, in which case
$$
\alpha(p^j)=\beta_{p,j}.
$$
\par
By definition of $Y_p(x)$ and of $\Lambda_d$, we get
$$
\sum_{m\in \friable_q(N,Q)}{ \alpha(m)\Lambda_m(\rho_{d}(f))}
=\sum_{p\leq Q}{(Y_p(\rho_p(f))-\beta_{p,0})},
$$
showing that~(\ref{eq-cor-1}) is indeed a special case
of~(\ref{eq-ls}).
\par
To prove~(\ref{eq-cor-2}), we use the ``amplification'' method of the
previous section. The basic observation is that if, for some prime
$p$, we have
\begin{equation}\label{eq-gap}
Y_p(\lambda_f(p))\leq \beta_{0,p}-\delta_p,
\end{equation}
then it follows that
$$
\sum_{1\leq i\leq s}{(-\beta_{p,i})X_i(\lambda_f(p))}\geq
\delta>0.
$$
\par
Now let $\xi_p$, for $p\leq Q$, be arbitrary auxiliary positive real
numbers, and let 
$$
\xi_d=\prod_{p\mid d}{\xi_p}
$$
for $d\mid P(Q)$, the product of all primes $p\leq Q$.
If~(\ref{eq-gap}) holds for all $p\leq Q$ coprime with $q$, then we
find by multiplying out that, for any integer $m\in \friable_q(N,Q)$,
i.e., such that
\begin{equation}\label{eq-cond-d}
  d\leq N,\quad  d\mid P(Q),\quad (d,q)=1,
  \quad d=p_1\cdots p_k,\text{ (say)},
\end{equation}
and for such $(\xi_p)$, we have
$$
\xi_d
\multsum_{1\leq j_1,\ldots,j_k\leq s}{ (-1)^k\beta_{p_1,j_1}\cdots
  \beta_{p_k,j_k} X_{j_1}(\lambda_f(p_1^{j_1}))
\cdots X_{j_k}(\lambda_f(p_k^{j_k}))} \geq
\xi_d\delta^{\omega(d)},
$$
which translates to
$$
\xi_d
\sum_{m\in S_d}{\alpha(m)\Lambda_m((\lambda_f(p))_{p\leq Q})} 
\geq
\xi_d\prod_{p\mid d}{\delta_p},
$$
where $m$ runs over the set $S_d$ of integers of the type
$$
m=\prod_{p\mid d}{p^{v_p(n)}},\quad\text{ with }\quad
1\leq v_p(n)\leq s,\quad \prod_{p\mid m}{p}\leq N,
$$
so $m\leq N^s$, and
$$
\alpha(m)=\prod_{p\mid d}{(-\beta_{p,v_p(m)})}.
$$
\par
Thus, summing over $d$ subject to~(\ref{eq-cond-d}), squaring, then
averaging over $f$ and applying~(\ref{eq-ls}), we find that the
probability
$$
\mathcal{P}=\proba_q\Bigl(Y_p(\lambda_f(p))\leq 
\beta_{p,0}-\delta_p\text{ for  all } p\leq Q\Bigr)
$$
satisfies
$$
\mathcal{P}\ll (1+N^sq^{-1})\frac{A_1}{B_1^2},
$$
where
$$
A_1=\sum_{d}{\xi_d^2
\sum_{m\in S_d}{|\alpha(m)|^2}}
=\sum_{d}{\xi_d^2\prod_{p\mid d}{\sigma_p^2}}
\quad
B_1=\sum_{d}{\xi_d\prod_{p\mid d}{\delta_p}}
$$
\par
Cauchy's inequality shows that $B_1^2\leq HA_1$, with equality if
$$
\xi_p=\frac{\delta_p}{\sigma_p^2},\text{ for all } p\leq Q,
$$
and the inequality above, with this choice, leads to
$$
\mathcal{P}\ll (1+N^sq^{-1})H^{-1},
$$
as desired.
\par
\begin{rem}
  If one tries to adapt, for instance, the standard proof
  in~\cite{sieve}, one encounters problems because the latter would
  (naively at least) involve the problematic expansion of a Dirac
  measure at a fixed $x\in [-2,2]$ in terms of Chebychev polynomials.
\end{rem}
% \begin{rem}
% In the case of classical sieves, this last proof is performed with
% $\eps=1$, since $n\mods{p}\notin \Omega_p$ implies that the
% characteristic function of $\Omega_p$ is zero at $n$,
% giving~(\ref{eq-gap}) with $\delta_p=\beta_{p,0}=|\Omega_p|p^{-1}$).
% \end{rem}
\par
Here is an easy application of~(\ref{eq-cor-2}), for illustration
(stronger results for that particular problem follow from the
inequality of Lau and Wu~\cite{lau-wu}, as will be explained with
other related results in a forthcoming joint work): it is well-known
that for $f\in S_k(q)^*$, the sequence of real numbers
$(\lambda_f(p))_p$ changes sign infinitely often, and there has been
some recent interest (see, e.g., the paper~\cite{iks} of Iwaniec,
Kohnen and Sengupta) in giving quantitative bounds on the first sign
change. We try instead to show that this first sign-change is quite
small on average over $f$ (compare with~\cite{duke-kowalski}): fix
$A>0$, and let
$$
S_{q,A}=\{f\in S_k(q)^*\,\mid\, \lambda_f(p)\leq 0\text{ for all }
p\leq (\log q)^A\}
$$
(any other combination of signs is permissible). This is a ``sifted
set'', and we claim that
$$
|S_{q,A}|\ll q^{1/2+1/(2A)+\eps}
$$
for any $\eps>0$, where the implied constant depends only on $k$ and
$\eps$. Since $S_k(q)^*$ is of size about $q$ (for fixed $k$), this is
a non-trivial bound for all $A>1$.  Moreover, to prove this bound, it
suffices to show
$$
\proba_q(S_{q,A})\ll q^{-1/2+1/(2A)+\eps}
$$
since we have the well-known upper bound $\mu_q(\{f\})\gg q^{-1-\eps}$
for any $\eps>0$ (see, e.g.,~\cite[p. 138]{ik}). 
\par
The sets $\Omega_p=]0,2]$ used in $S_{q,A}$ are not exactly in the
form~(\ref{eq-spec-omega}), so we use some smoothing: we claim there
exists a real polynomial $Y$ of degree $s=2$ such that
\begin{equation}\label{eq-pol-sign}
  Y(x)\leq \mathrm{sgn}(x),\quad\text{ for all } x\in [-2,2],\quad\text{ and }
  \quad
  \beta_0=\int_{-2}^2{Yd\mu_{ST}}>-1.
\end{equation}
\par
Assuming such a polynomial is given, we observe that
$$
\lambda_f(p)\geq 0\Rightarrow Y(\lambda_f(p))\leq
-1=\beta_0-\delta,
$$
for some fixed $\delta>0$. Therefore, by~(\ref{eq-cor-2}) with
$N=q^{1/s}$, we get for all $Q$ that
$$
\proba_q(\lambda_f(p)\leq 0\text{ for } p\leq Q)
\ll H^{-1}
$$
where 
$$
H=\sum_{m\in \friable_q(q^{1/s},Q)}{\gamma^{\omega(m)}
},
\quad\quad
\text{ with }\quad\quad \gamma=\frac{(\beta_0+1)^2}
{\beta_1^2+\beta_2^2},
$$
and the implied constant depends only on $k$. By assumption, we have
$\gamma>0$, and an easy lower bound for $H$ follows in the range of
interest simply from bounding $\gamma^{\omega(m)}\gg m^{-\eps}$ and
using known results on the cardinality of $\friable_q(y,(\log q)^A)$:
we have
$$
\sum_{m\in \friable_q(q^{1/s},(\log q)^{A})}{\gamma^{\omega(m)}}
\gg q^{s^{-1}(1-A^{-1})-\eps},
$$
for any $\eps>0$, the implied constant depending only on $A$, the
choice of $Y$ and $\eps$. This clearly gives the result, and it only
remains to exhibit the polynomial $Y$.  One can check easily that
$$
Y(x)=-\frac{3}{4}X_0(x)+\frac{1}{2}X_1(x)+\frac{1}{4}X_2(x)
=-1+\frac{x}{2}+\frac{x^2}{4}
$$
does the job (see its graph); the numerical values of $\beta_0$,
$\delta$ and $\gamma$ are given by
$$
\beta_0=-\frac{3}{4},\quad\quad \delta=\frac{1}{4},
\quad\quad
\beta_1^2+\beta_2^2=\frac{5}{16},\quad\quad
\gamma=\frac{4}{5}.
$$
\par
\begin{figure}[ht]
\centering
\includegraphics[width=3.5in]{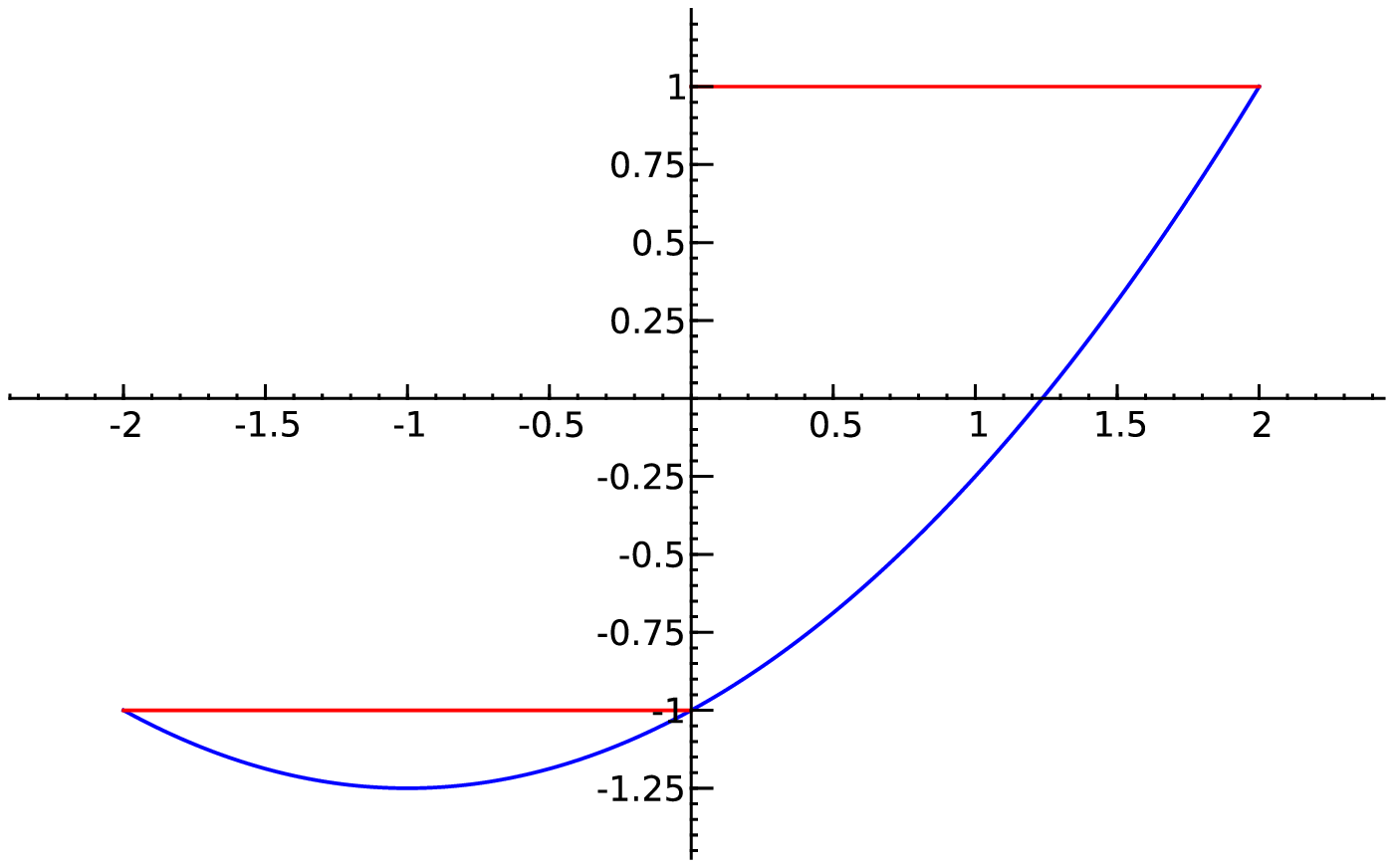}
\caption{}
\end{figure}

\begin{remark}
See the letter of Serre in the Appendix of~\cite{shahidi} for previous
examples showing how to use limited information towards the Sato-Tate
conjecture to prove distribution results for Hecke eigenvalues (of a
fixed modular form).
\end{remark}

%%%%%%%%%%%%%%%%%%%%%%%%%%%%%%%%%%%%%%%

\end{document}